\documentclass[letterpaper, 10 pt, journal, twoside]{IEEEtran}

\usepackage{booktabs}
\usepackage{cite}
\usepackage[hidelinks]{hyperref}

\usepackage{enumitem}

\usepackage{amsthm}
\usepackage[cmex10]{amsmath}
\newtheorem{theorem}{Theorem}
\newtheorem{lemma}{Lemma}

\newtheorem{assumption}{Assumption}
\newtheorem{proposition}{Proposition}

\theoremstyle{remark}
\newtheorem{remarkx}{Remark}

\newcommand\oprocendsymbol{\hbox{\small $\blacksquare$}}
\newcommand\oprocend{\relax\ifmmode\else\unskip\hfill\fi\oprocendsymbol}

\newenvironment{remark}{\remarkx}{\oprocend\endremarkx}

\usepackage[capitalize,nameinlink,compress]{cleveref}

\crefname{figure}{Figure}{Figures}
\crefname{equation}{}{}
\crefname{assumption}{Assumption}{Assumptions}

\usepackage{amsfonts}
\usepackage{amssymb}

\DeclareMathOperator{\cl}{cl}
\DeclareMathOperator{\interior}{int}

\newcommand{\bbR}{\mathbb{R}}
\newcommand{\bbB}{\mathbb{B}}

\newcommand{\calM}{\mathcal{M}}
\newcommand{\calS}{\mathcal{S}}
\newcommand{\calL}{\mathcal{L}}
\newcommand{\calU}{\mathcal{U}}
\newcommand{\calA}{\mathcal{A}}
\newcommand{\calC}{\mathcal{C}}
\newcommand{\calK}{\mathcal{K}}

\renewcommand{\o}[1]{\overline{#1}}

\renewcommand{\d}[1]{\dot{#1}}

\usepackage{tikz}
\usepackage{tkz-euclide}
\usepackage[european]{circuitikz}
\usepackage[caption=false,font=footnotesize]{subfig}

\usetikzlibrary{matrix}
\usetikzlibrary{arrows}
\usetikzlibrary{calc}
\usetikzlibrary{fit}
\usetikzlibrary{shapes.callouts,shapes.arrows}

\tikzstyle{block} = [draw, rectangle, thick,minimum height=1em,minimum width=1em]
\tikzstyle{smallsum} = [draw,circle,inner sep=0mm,minimum size=3mm]
\tikzstyle{sum} = [draw,circle,inner sep=0mm,minimum size=6mm]
\tikzstyle{branch} = [draw,circle,inner sep=0.5mm,fill=black]
\tikzstyle{none} = [draw=none]
\tikzstyle{connector} = [->,thick]
\tikzstyle{line} = [thick]

\tikzset{%
  saturation block/.style={%
    draw, thick,
    path picture={
      \pgfpointdiff{\pgfpointanchor{path picture bounding box}{north east}}%
        {\pgfpointanchor{path picture bounding box}{south west}}
      \pgfgetlastxy\x\y
      \tikzset{x=\x*.4, y=\y*.4}
      \draw (-.9,0) -- (.9,0) (0,-.9) -- (0,.9);
      \draw (-.9,-.6) -- (-.6,-.6) -- (.6,.6) -- (.9,.6);
      \node[text width=.1cm] at (-.35,.55) {\scriptsize $P_\calC$};
    }
  }
}

\title{On the Differentiability of Projected Trajectories and the Robust Convergence of Non-convex Anti-Windup Gradient Flows
}

\author{Adrian Hauswirth$^\dagger$, Florian D\"orfler$^\dagger$,
Andrew Teel$^\ddagger$
\thanks{$^\dagger$A. Hauswirth and F. D\"orfler are with the Department of Information Technology and Electrical Engineering, ETH Z\"urich,
                Zurich, Switzerland. Email:
                {\tt\small \{hadrian, dorfler\}@ethz.ch}.
}%
\thanks{$^\ddagger$A. Teel is with the is with the Department of Electrical and Computer Engineering, University of California, Santa Barbara, CA, USA.
E-mail: {\tt\small teel@ece.ucsb.edu}.
}%
\thanks{This work was supported by ETH Zurich funds, SNF AP Energy grant \#160573, mobility grant \#160573/2, and AFOSR grant FA9550-18-1-0246.
}%
}

\begin{document}
\maketitle
\thispagestyle{empty}
\pagestyle{empty}

\begin{abstract}
    This paper concerns a new class of discontinuous dynamical systems for constrained optimization. These dynamics are particularly suited to solve nonlinear, non-convex problems in closed-loop with a physical system. Such approaches using feedback controllers that emulate optimization algorithms have recently been proposed for the autonomous optimization of power systems and other infrastructures. In this paper, we consider feedback gradient flows that exploit physical input saturation with the help of anti-windup control to enforce constraints. We prove semi-global convergence of ``projected'' trajectories to first-order optimal points, i.e., of the trajectories obtained from a pointwise projection onto the feasible set. In the process, we establish properties of the directional derivative of the projection map for non-convex, prox-regular sets.
\end{abstract}


\section{Introduction}
\IEEEPARstart{W}{hen} a trajectory of a continuous-time dynamical system is projected pointwise on a closed convex set, one obtains a ``projected'' trajectory (see \cref{fig:proj_traj}) that is in general not differentiable nor does it satisfy a particular law of motion. Nevertheless, these projected trajectories have interesting properties in their own right, but seem to have been largely ignored by the research community.

One particularly interesting context in which projected trajectories occur is a control loop with a saturated integrator. In this case, one can interpret the saturated control input as a signal projected on a set of feasible inputs, thus resulting in a projected trajectory. However, the main complexity lies in the fact that the unsaturated control signal is itself coupled with the saturated signal through feedback.

In this context, the cascade of an integrator and a saturation element is well-known to be prone to integrator windup which can seriously degrade performance. Anti-windup schemes are an effective and well-established solution to mitigate this problem~\cite{zaccarianModernantiwindupsynthesis2011,tarbouriechAntiwindupdesignoverview2009}. Moreover, the authors have recently shown that high-gain anti-windup schemes~\cite{hauswirthImplementationProjectedDynamical2020,hauswirthAntiWindupApproximationsOblique2020} applied to integral controllers can also be used to approximate projected dynamical systems~\cite{hauswirthProjectedDynamicalSystems2018,nagurneyProjectedDynamicalSystems1996, aubinDifferentialInclusionsSetValued1984}.

These facts are particularly useful in the context of feedback-based optimization which has recently garnered a lot of interest for applications such as the real-time control and optimization of power systems~\cite{molzahnSurveyDistributedOptimization2017}, communication networks~\cite{lowInternetCongestionControl2002}, and other infrastructure systems. Feedback-based optimization aims at designing feedback controllers that can steer a (stable) physical system to a steady state that solves a well-defined, but partially unknown, constrained optimization problem, for instance by designing feedback controllers to implement gradient~\cite{hauswirthProjectedGradientDescent2016,colombinoRobustnessGuaranteesFeedbackbased2019,colombinoOnlineOptimizationFeedback2019} or saddle-point flows~\cite{tangRealTimeOptimalPower2017,dallaneseOptimalPowerFlow2018,lawrenceOptimalSteadyStateControl2018} as a closed-loop behavior.

One aspect of feedback-based optimization is the exploitation of physical saturation to enforce (unknown or time-varying) input constraints. Within this context, we study in this paper a discontinuous dynamical system that arises as a feedback loop based on gradient flow, subject to saturation, and augmented with an anti-windup scheme (see \cref{fig:aw1}).

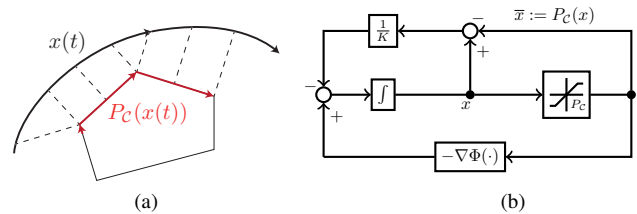
\begin{figure}[htb]
	\centering
	\subfloat[\label{fig:proj_traj}]{
		\centering
		\includegraphics[width=.40\columnwidth]{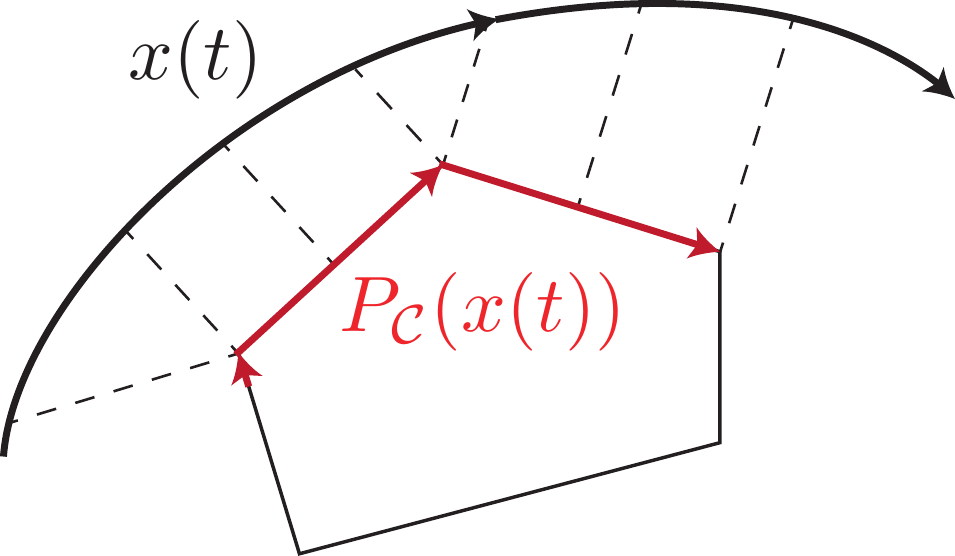}
	}
	\subfloat[\label{fig:aw1}]{
		\centering
		\begin{tikzpicture}[thick, scale = 0.65, every node/.style={scale=0.65}]
			\matrix[ampersand replacement=\&, row sep=0.3cm, column sep=.48cm] {

				\& \node[block] (awgain) {$\frac{1}{K}$}; \&
				\node[smallsum] (awsum) {};    \\

				\node[smallsum] (aw_inf) {};
				\& \node[block] (int) {$\int$};
				\& \node[branch](br1) {};
				\& \node[saturation block, minimum height=2.7em,minimum width=2.7em] (sat) {};
				\& \node[branch](br2){}; \\

				\node[none](edge) {};
				\& \& \node[block] (grad) {$- \nabla \Phi(\cdot) $}; \& \& \& \&
				\node[none](end) {};   \\
			};

			\draw[connector] (awgain.west)-|(aw_inf.north) node[at end, left]{$-$};
			\draw[connector] (aw_inf.east)--(int.west);
			\draw[line] (int.east)--(br1.west) node[at end, below] {$x$};
			\draw[connector] (br1.east)--(sat.west);
			\draw[line] (sat.east)--(br2.west);
			\draw[connector] (br2.south)|-(grad.east) {};
			\draw[line] (grad.west)-|(edge.center);
			\draw[connector] (edge.center)--(aw_inf.south) node[near end, right] {$+$};

			\draw[connector] (br2.north)|-(awsum.east) node[at end, above] {$-$} node[near end, above ]{$\overline{x} := P_\calC(x)$};
			\draw[connector] (br1.north)--(awsum.south) node[near end, right] {$+$};
			\draw[connector] (awsum.west)--(awgain.east);
		\end{tikzpicture}
	}
	\caption{a) Projected trajectory; b) Gradient feedback loop with anti-windup}
\end{figure}

\subsection{Simplified Problem Formulation}\label{sec:simp_prob_form}

Consider a closed convex set $\calC \subset \bbR^n$ and let $P_\calC$ denote the Euclidean minimum norm projection onto $\calC$, i.e.,  $P_\calC(x) = \arg \min_{y \in \calC} \| y - x \|$. Further, let $\Phi: \bbR^n \rightarrow \bbR$ be convex, continuously differentiable in a neighborhood of $\calC$, and have compact sublevel sets. We consider the dynamical system
\begin{align}\label{eq:aw_sys}
	\dot x = - \nabla \Phi(P_\calC(x)) - \tfrac{1}{K} \left( x - P_\calC(x) \right) \, ,
\end{align}
where $K > 0$ is fixed. Note in particular that $\nabla \Phi$ is evaluated at the point $P_\calC(x)$. \cref{fig:aw1} illustrates~\eqref{eq:aw_sys} as a feedback loop. We want to show that $t \mapsto P_\calC(x(t))$, where $x$ is a solution of~\eqref{eq:aw_sys}, converges to an optimizer of the problem
\begin{align}\label{eq:basic_prob}
	\text{minimize} \quad \Phi(x) \quad \text{subject to} \quad x \in \calC \, .
\end{align}

We call~\eqref{eq:aw_sys} an \emph{anti-windup approximation} of a projected gradient flow, because the term $\frac{1}{K}(x - P_\calC(x))$ can be realized by an anti-windup scheme as shown in \cref{fig:aw1}. Furthermore, in~\cite{hauswirthAntiWindupApproximationsOblique2020} it was shown that the solutions of~\eqref{eq:aw_sys} converge uniformly to the trajectory of a projected gradient flow~\cite{hauswirthProjectedGradientDescent2016} as $K\rightarrow 0^+$. Simulations and numerical examples for systems of the form~\eqref{eq:aw_sys} can be found in~\cite{hauswirthAntiWindupApproximationsOblique2020}.

In~\cite{hauswirthImplementationProjectedDynamical2020} it has also been noted that the system~\eqref{eq:aw_sys} bears similarities with the gradient flow
\begin{align}\label{eq:aw_pen_grad}
	\dot x = - \nabla \tilde{\Phi}(x) = - \nabla \Phi(x) - \tfrac{1}{K} \left( x - P_\calC(x) \right) \, .
\end{align}
where $\tilde{\Phi}(x) := \Phi(x) + \frac{1}{2K} d_\calC^2(x)$ and $d_\calC$ denotes the point-to-set distance to $\calC$. Namely, $\tilde{\Phi}$ is a cost function augmented with a term penalizing the distance from the feasible set $\calC$.

However, the inconspicuous difference between~\eqref{eq:aw_sys} and~\eqref{eq:aw_pen_grad} in the argument of $\nabla \Phi$ leads to two important contrasts:

First, if $x^\star$ is an equilibrium of~\eqref{eq:aw_sys}, then $P_\calC(x^\star)$ is a optimizer of~\eqref{eq:basic_prob}~\cite[Prop. 4]{hauswirthImplementationProjectedDynamical2020}. This is not the case for equilibria of~\eqref{eq:aw_pen_grad}; equilibria of~\eqref{eq:aw_pen_grad} are minimizers of $\tilde{\Phi}$ but not necessarily optimizers of~\eqref{eq:basic_prob}.
Second, convergence to the set of global minimizers of~$\tilde{\Phi}$ can be easily established for~\eqref{eq:aw_pen_grad}. However, proving convergence of solutions of~\eqref{eq:aw_sys} to optimizers of~\eqref{eq:basic_prob} is more challenging. In~\cite[Th. 6.4]{hauswirthAntiWindupApproximationsOblique2020} convergence was shown  under strong convexity and Lipschitz continuity of $\nabla \Phi$, and for small enough $K$.

\subsection{Contributions}

In this paper we show that (projected) trajectories of~\eqref{eq:aw_sys} converge to first-order optimal points of~\eqref{eq:basic_prob}, as postulated above, under the following weakened assumptions:
\begin{enumerate}
	\item We do not assume convexity of $\Phi$. Instead, we simply require differentiability and compact sublevel sets (on $\calC$) which are the minimal assumptions for standard gradient flows to be well-defined and convergent.

	\item We do not require convexity of $\calC$. Instead, we consider the class of (non-convex) prox-regular sets, which, roughly speaking, are those sets for which the projection $P_\calC$ is single valued in a neighborhood of $\calC$.
\end{enumerate}

In this general setup convergence is ``semi-global'', i.e., for every compact set of initial conditions, one can find $K$ small enough to guarantee convergence. However, if $\calC$ is convex, we show that~\eqref{eq:aw_sys} is globally convergent for any $K > 0$.

Hence, our results in this paper not only strengthen~\cite[Th. 6.4]{hauswirthAntiWindupApproximationsOblique2020}, but are also based on a different approach. In particular, as a by-product of our analysis, we establish properties of the directional derivative of $P_\calC$ for prox-regular sets. These results are potentially useful outside the scope of our problem for the study of projected trajectories.

\subsection{Solution Approach \& Related Work}

To show that solutions of~\eqref{eq:aw_sys} converge to optimizers of~\eqref{eq:basic_prob} we apply an invariance argument for which we need that $x \mapsto \Phi(P_\calC(x))$ is non-increasing along trajectories of~\eqref{eq:aw_sys}. However, to evaluate the Lie derivative of $\Phi \circ P_\calC$, $P_\calC$ needs to admit a derivative.

The differentiability of $P_\calC$ has been studied extensively, albeit---to the best of the authors' knowledge---only for convex sets $\calC$. Even if $\calC$ is convex, $P_\calC$ is in general not differentiable unless $\calC$ has a smooth boundary~\cite{fitzpatrickDifferentiabilityMetricProjection1982}.
Further, $P_\calC$ is not generally directionally differentiable~\cite{shapiroDirectionallynondifferentiablemetric1994,kruskalTwoConvexCounterexamples1969} unless \emph{second-order regularity} assumptions on $\calC$ are satisfied~\cite{shapiroDifferentiabilityPropertiesMetric2016,bonnansSensitivityAnalysisOptimization1998}. An up-to-date review of this subject including detailed examples can also be found in \cite{veermanNavigatingConvexSets2019}.
We avoid these technicalities because we require directional differentiability only along a trajectory (c.f. \cref{lem:exist}).

\subsection{Organization}

In \cref{sec:prelim} we fix the notation and recall relevant notions from variational geometry. \cref{sec:dir_deriv} studies directional derivatives of projection maps for prox-regular sets. In \cref{sec:main_prob} we state our main problem and results for which the proofs can be found in \cref{sec:main_proof}. Finally, \cref{sec:concl} summarizes our findings and discusses open questions.

\section{Preliminaries}\label{sec:prelim}

\subsection{Notation}
We consider $\bbR^n$ with the usual inner product $\left\langle \cdot, \cdot \right\rangle$ and 2-norm $\| \cdot \|$. The closed (open) unit ball of appropriate dimension is denoted by $\bbB$ ($\interior \bbB$). For a sequence $\{K_n \}$, the notation $K_n \rightarrow 0^+$ implies that that $K_n > 0$ for all $n$ and ${\lim}_{n \rightarrow \infty} K_n = 0$.
We use the standard definitions of \emph{outer/inner semicontinuity}, \emph{local boundedness}, etc. for set-valued maps $F: \bbR^n \rightrightarrows \bbR^m$~\cite[Ch. 5]{goebelHybridDynamicalSystems2012}.
Given a set $\calC \subset \bbR^n$, its \emph{closure} is denoted by $\cl{\calC}$. The \emph{distance function} is defined as $d_\calC(x) := \inf_{\tilde{x} \in \calC} \| x - \tilde{x} \|$. The \emph{projection mapping} $P_\calC : \bbR^n \rightrightarrows \calC$ is given by $P_\calC(x) := \{ \tilde{x} \in \calC \, | \, \| x - \tilde{x} \| = d_{\calC}(x) \}$.
We use the standard definition of \emph{(Bouligand) tangent cone of $\calC$ at $x \in \calC$} which we denote by $T_x \calC$~\cite[Ch. 6]{rockafellarVariationalAnalysis2009}. The set $\calC$ is \emph{Clarke regular} (or \emph{tangentially regular}) if it is closed and $x \mapsto T_x \calC$ is inner semicontinuous~\cite[Cor. 6.29]{rockafellarVariationalAnalysis2009}. If $\calC$ is Clarke regular, $N_x \calC$ denotes the normal cone of $\calC$ at $x \in \calC$ (which is defined as the polar cone of $T_x \calC$).

\subsection{Prox-Regular Sets}

Consider a Clarke regular set $\calC \subset \bbR^n$ and $x \in \calC$. Given $\alpha > 0$, a normal vector $\eta \in N_x \calC$ is \emph{$\alpha$-proximal} if $\left\langle \eta, y - x \right\rangle \leq \alpha \| \eta \| \| y - x \|^2$ for all $y \in \calC$ and $\calC$ is \emph{$\alpha$-prox-regular at $x$} if all normal vectors at $x$ are $\alpha$-proximal.
A set is \emph{$\alpha$-prox-regular} if it is $\alpha$-prox-regular at all $x \in \calC$.

As a specific example, note that every closed convex set is $\alpha$-prox-regular for any $\alpha > 0$. Furthermore, every set of the form $\calC = \{ x \, | \, g(x) \leq 0 \}$, where $g: \bbR^n \rightarrow \bbR^m$ has a globally Lipschitz derivative and constraint qualifications apply, is $\alpha$-prox-regular for some $\alpha > 0$~\cite[Ex.~7.7]{hauswirthProjectedDynamicalSystems2018}.

The following key properties of prox-regular sets are taken from \cite[Th. 2.2 \& Prop. 2.3]{adlyPreservationProxRegularitySets2016}.

\begin{proposition}\label{prop:prox}
	If $\calC \subset \bbR^n$ is $\alpha$-prox-regular, then for any $x \in \calC + \frac{1}{2\alpha} \interior \bbB$ the set $P_\calC(x)$ is a singleton and $d^2_\calC$ is differentiable with $\nabla (d^2_\calC(x)) = 2(x - P_\calC(x))$.
\end{proposition}

\begin{lemma}\label{lem:proj_preimag}
	If $\calC \subset \bbR^n$ is $\alpha$-prox-regular, then $P_\calC(x + v) = x$ holds for every $x \in \calC$ and all $v \in N_x \calC \cap \frac{1}{2\alpha} \interior \bbB$. Further, for all $y \in \calC + \frac{1}{2\alpha} \interior \bbB$, we have $y - P_\calC(y) \in N_{P_\calC(y)} \calC$.
\end{lemma}

\begin{lemma}\label{lem:lip_prox}
	Let $\calC \subset \bbR^n$ be $\alpha$-prox-regular, then the projection $x \mapsto P_\calC(x)$ is locally Lipschitz on $\calC + \tfrac{1}{2\alpha} \interior \bbB$.
\end{lemma}

Another crucial property of prox-regular sets is that the normal cone mapping $x \mapsto N_x \calC$ admits a \emph{hypomonotone localization}~\cite[Ex. 13.38]{rockafellarVariationalAnalysis2009}. We exploit this property through the following lemma which, in contrast to~\cite[Ex. 13.38]{rockafellarVariationalAnalysis2009}, quantifies the hypomonotonicity in terms of $\alpha$.

\begin{lemma}\label{lem:hypomonot}
	Let $\calC \subset \bbR^n$ be $\alpha$-prox-regular. Then, for all $x, x' \in \calC$, $\eta \in N_x \calC \cap \bbB$, and $\eta' \in N_{x'} \calC \cap \bbB$, we have
	\begin{align*}
		\left\langle \eta' - \eta, x'- x \right\rangle \geq - 2 \alpha \| x' - x \|^2 \, .
	\end{align*}
\end{lemma}

\begin{proof}
	Since $0 \leq \| \eta\| \leq 1$ and $0 \leq \| \eta' \| \leq 1$ it follows from the definition of prox-regularity that
	\begin{align*}
		\left\langle \eta , x' - x \right\rangle    & \leq \alpha \| x' - x \|^2 \\
		\left\langle - \eta' , x' - x \right\rangle & \leq \alpha \| x' - x \|^2 \, .
	\end{align*}
	Adding up both inequalities yields the desired result.
\end{proof}

\subsection{Dynamical Systems \& Invariance Principle}

In general, we understand a dynamical system to be defined by a differential inclusion  (e.g.,~\cite{filippovDifferentialEquationsDiscontinuous1988}) of the form
\begin{align}\label{eq:incl}
	\dot x \in F(x) \, ,
\end{align}
where $F: \bbR^n \rightrightarrows \bbR^n$ is outer semicontinuous and locally bounded, and $F(x)$ is convex and non-empty for all $x \in \bbR^n$. A map $x:[0, T] \rightarrow \bbR^n$ for some $T>0$ is a \emph{solution of~\eqref{eq:incl}} if $x$ is absolutely continuous and $\dot x(t) \in F(x(t))$ holds for almost all $t \in [0, T]$. Existence of solutions for any $x(0) \in \bbR^n$ is guaranteed under the given assumptions on $F$. A \emph{complete} solution is a map $x: [0, \infty) \rightarrow \bbR^n$ such that the restriction to any compact subinterval $[0, T]$ is a solution.

Throughout the paper, we will mostly encounter differential inclusions that reduce to a continuous differential equation on an invariant subset of $\bbR^n$. In other words, on a subset $\calA$ of $\bbR^n$, $F$ in~\eqref{eq:incl} is a single-valued, continuous map and, moreover, any solution of~\eqref{eq:incl} starting in $\calA$ remains in $\calA$. In this case, a solution $x: [0, T] \rightarrow \calA$ to~\eqref{eq:incl} is continuously differentiable and satisfies $\dot{x}(t) = F(x(t))$ for all $t \in [0, T]$.

We require the following standard invariance principle for differential inclusions (see also \cite[Th. 2.10]{ryanIntegralInvariancePrinciple1998} and~\cite{bacciottiNonpathologicalLyapunovfunctions2006}):

\begin{theorem}\cite[Th. 8.2]{goebelHybridDynamicalSystems2012 }\label{thm:invar}
	Consider a continuous function $V: \bbR^n \rightarrow \bbR$, any functions $u: \bbR^n \rightarrow [- \infty, \infty]$, and a set $\calU \subset \bbR^n$ such that $u(x) \leq 0$ for every $x \in \calU$ and such that the growth of $V$ along solutions of~\eqref{eq:incl} is bounded by $u$ on $\calU$, i.e., any solution $x: [0, T] \rightarrow \calU$ of~\eqref{eq:incl} satisfies
	\begin{align*}
		V(x(t_1)) - V(x(t_0)) \leq \int_{t_0}^{t_1} u ( x(\tau)) d\tau \, .
	\end{align*}
	Let a complete and bounded solution $x$ of~\eqref{eq:incl} be such that $x(t) \in \calU$ for all $t \in [0, \infty)$. Then, for some $r \in V(\calU)$, $x$ approaches the nonempty set that is the largest weakly invariant subset of $V^{-1}(r) \cap \calU \cap \cl u^{-1}(0)$.
\end{theorem}

\section{Directional Derivatives of Projection Maps and Projected Trajectories}\label{sec:dir_deriv}

Next, given a closed set $\calC \subset \bbR^n$, we establish properties of the directional derivative of $P_\calC$.
Recall that the directional derivative of $P_\calC$ at $x \in \bbR^n$ in direction $v \in \bbR^n$ is defined as
\begin{align}\label{eq:def_direct}
	D P_\calC(x; v) := \underset{h \rightarrow 0^+}{\lim} \, \frac{ P_\calC (x + hv) - P_\calC(x)}{h} \, .
\end{align}

The classical result~\cite[Prop. III.5.3.5]{hiriart-urrutyConvexAnalysisMinimization1996} states that for convex $\calC$, $D P_\calC (x; v)$ exists for all $x \in \calC$  and all $v \in \bbR^n$ and is given as the projection of $v$ onto the tangent cone at $x$. Its generalization to $\alpha$-prox-regular sets is straightforward.

\begin{lemma}\label{lem:proj_dir_deriv}
	Let $\calC \subset \bbR^n$ be $\alpha$-prox-regular for some $\alpha > 0$. Then, $D P_\calC(x; v)$ exists for all $x \in \calC$ and all $v \in \bbR^n$ and is given by
	\begin{align*}
		D P_\calC(x; v) = P_{T_x \calC} \left( v  \right) = \lim_{h \rightarrow 0^+} \frac{P_\calC(x + h v) - x}{h} \, .
	\end{align*}
\end{lemma}

Characterizing the $D P_\calC(x; v)$ at $x \notin \calC$ is harder and directional differentiability is in general not guaranteed (see~\cite{kruskalTwoConvexCounterexamples1969,shapiroDirectionallynondifferentiablemetric1994}). However, the forthcoming \cref{lem:proj_traj_basic} guarantees that, along an absolutely continuous trajectory, the directional derivative of $P_\calC$ exists for almost all~$t$.

Assuming that $D P_\calC (x; v)$ exists, one can establish various properties. First of all, it immediately follows from the definition of the tangent cone that $D P_\calC (x; v)$ is \emph{viable}:

\begin{lemma}\label{lem:viab}
	If $\calC \subset \bbR^n$ is $\alpha$-prox-regular, $x \in \calC + \frac{1}{2\alpha} \interior \bbB$, $v \in \bbR^n$, and if $D P_\calC(x; v)$ exists, then $D P_\calC(x; v) \in T_{P_\calC(x)} \calC$.
\end{lemma}

The next two lemmas exploit basic properties of $P_\calC$.

\begin{lemma}\label{lem:proj_traj_mono}
	Consider an $\alpha$-prox-regular set $\calC \subset \bbR^n$ and let $x \in \calC + \frac{1}{2\alpha} \interior \bbB$ and $v \in \bbR^n$ be such that $\o{v} := D P_\calC(x; v)$ exists. Then, we have $\left\langle v, \o{v} \right\rangle \geq 0$.
\end{lemma}

\begin{proof}
	Recall that for a closed set $\calC$, the projection $P_\calC$ is monotone~\cite[Cor.~12.20]{rockafellarVariationalAnalysis2009}.
	It follows that this property also holds in the limit by continuity of $P_\calC$ (\cref{lem:lip_prox}), i.e.,
	\begin{align*}
		\left\langle \o{v}, v \right\rangle
		= \lim_{h \rightarrow 0^+} \tfrac{ \left\langle P_\calC(x + hv) - P_\calC(x) , (x + hv) - x \right\rangle}{h^2} \geq 0 \, .
        \tag*{\qedhere}
	\end{align*}
\end{proof}

\begin{lemma}\label{lem:orth_prop}
	Let $\calC \subset \bbR^n$ be $\alpha$-prox-regular, $x \in \calC + \tfrac{1}{2\alpha} \interior \bbB$, $v \in \bbR^n$, and assume that $\o{v} := D P_\calC(x; v)$ exists. Then, it holds that $\left\langle \o{v}, x - P_\calC(x) \right\rangle  = 0$.
\end{lemma}

\begin{proof}
	Define the map $\phi(h) := x + hv$ for all $h \geq 0$.
	Using \cref{prop:prox} and the chain rule, we know that
	\begin{align*}
		\nabla \left(d_\calC^2 \circ \phi \right) |_{h = 0} = 2 \left\langle v, x - P_\calC(x) \right\rangle \, .
	\end{align*}
	On the other hand, we can apply the chain rule to $d_\calC^2(\phi(h)) = \| \phi(h) - P_\calC(\phi(h)) \|^2$ to arrive at
	\begin{align*}
		\nabla \left( \| \phi(h) - P_\calC(\phi(h)) \|^2 \right)|_{h = 0}
		= 2 \left \langle  v - \o{v}, x - P_\calC(x) \right\rangle \, .
	\end{align*}
	The difference of the expressions yields the result.
\end{proof}

\cref{lem:viab,lem:orth_prop} yield that $DP_\calC(x; v)$, if it exists, lies in $\calK(x) := T_x \calC \cap \{ v \, | \, \left\langle v, x - \o{x} \right\rangle = 0 \}$ which is known as the \emph{critical cone at $x$}. This observation is in agreement with~\cite{shapiroDifferentiabilityPropertiesMetric2016} and generalizes this insight from convex to prox-regular sets.

For the next crucial lemma we exploit the hypomonotone localization of $x \mapsto N_x \calC$ according to \cref{lem:hypomonot}.

\begin{lemma}\label{lem:proj_equi}
	Let $\calC \subset \bbR^n$ be $\alpha$-prox-regular, $x \in \calC + \tfrac{1}{2\alpha} \interior \bbB$, $v \in \bbR^n$, and assume that $\o{v} := D P_\calC(x; v)$ exists. Then,
	\begin{align*}
		\left\langle v, \o{v} \right\rangle  = 0 \quad \Longleftrightarrow \quad \o{v} = 0 \, .
	\end{align*}
\end{lemma}

\begin{proof}
	($\Leftarrow$) is trivial.
    For ($\Rightarrow$), consider $h > 0$ such that $x_h := x + hv \in \calC + \tfrac{1}{2\alpha} \interior \bbB$. Further, let $\o{x}_h = P_\calC(x_h)$ and $\o{x} := P_\calC(x)$, as well as $\eta := x - \o{x}$ and $\eta_h = x_h - \o{x}_h$. Recall that $\eta \in N_{\o{x}} \calC$ and $\eta_h \in N_{\o{x}_h} \calC$ (\cref{lem:proj_preimag}). Using these facts and the definition of $\o{v} = D P_\calC(x; v)$ in~\eqref{eq:def_direct} we can write
	\begin{align*}
		\left\langle v, \o{v} \right\rangle
		 & = \lim_{h \rightarrow 0^+} \tfrac{1}{h^2} \left\langle x_h - x , \o{x}_h - \o{x} \right\rangle                                   \\
		 & = \lim_{h \rightarrow 0^+} \tfrac{1}{h^2} \left\langle \o{x}_h + \eta_h - \o{x} + \eta, \o{x}_h - \o{x} \right\rangle            \\
		 & = \| \o{v} \|^2 + \underset{h \rightarrow 0^+}{\lim} \tfrac{1}{h} \left\langle \eta_h - \eta, \o{x}_h - \o{x} \right\rangle \, .
	\end{align*}

	Since, by assumption, $x \in \calC + \tfrac{1}{2\alpha} \interior \bbB$, there exists $\epsilon > 0$ such that $x, x_h \in \calC + \frac{1}{2\alpha + \epsilon}$ for small enough $h$. Therefore, $\| \eta\|$ and $\| \eta_h\|$ are both upper bounded by $\frac{1}{2\alpha + \epsilon}$.

	To apply \cref{lem:hypomonot} we rescale $\hat{\eta}:= (2\alpha + \epsilon) \eta$ and $\hat{\eta}_h := (2\alpha + \epsilon)\eta_h$ which satisfy $\hat{\eta}, \hat{\eta}_h \in \bbB$. It follows that
	\begin{align*}
		\left\langle v, \o{v} \right\rangle &
		= \| \o{v} \|^2
		+ \underset{h \rightarrow 0^+}{\lim} \frac{1}{(2 \alpha + \epsilon) h^2} \underbrace{\left\langle \o{x}_h - \o{x}, \hat{\eta}_h-  \hat{\eta}\right\rangle}_{\geq - 2 \alpha \| \o{x}_h - \o{x} \|^2 } \\
	   & \geq  \| \o{v} \|^2
		- \underset{h \rightarrow 0^+}{\lim} \frac{2\alpha }{(2 \alpha + \epsilon) h^2} \|\o{x}_h - \o{x} \|^2                                                                                                \\
		                                    & = \left(1 - \tfrac{2\alpha}{2\alpha + \epsilon} \right)\| \o{v}\|^2 \, .
	\end{align*}
	Since $\epsilon > 0$, we have $\tfrac{2\alpha}{2\alpha + \epsilon} < 1$ and thus $\left\langle v, \o{v} \right\rangle = 0$ implies that ${\o{v}} = 0$ which completes the proof.
\end{proof}

If $\calC$ is closed convex, Lemmas~\ref{lem:viab}-\ref{lem:proj_equi} simplify to the following facts (see also~\cite{fitzpatrickDifferentiabilityMetricProjection1982,shapiroDifferentiabilityPropertiesMetric2016}, and others):
\begin{lemma}\label{lem:dirderiv_cvx}
	Let $\calC \subset \bbR^n$ be closed convex and let $x \in \bbR^n$ and $v \in \bbR^n$ be such that $\o{v} := DP_\calC(x;v)$ exists. Then,
	\begin{enumerate}[label=(\roman*)]
		\item $\o{v} \in T_{P_\calC(x)} \calC \cap \{ w \, | \, \left\langle w, x - P_\calC(x) \right\rangle = 0 \}$,
		\item $\left\langle v, \o{v} \right\rangle \geq 0$, and
		\item $\left\langle v, \o{v} \right\rangle = 0 \quad \Longleftrightarrow \quad \o{v} = 0$.
	\end{enumerate}

\end{lemma}

\subsection{Projected Trajectories}

As mentioned before, establishing directional differentiability of $P_\calC$, i.e., the existence of $DP_\calC(x; v)$ for all $x \in \bbR^n$ and all directions $v \in \bbR^n$ is a major challenge and in general not possible without additional assumptions on $\calC$.
For our purposes, we do not require directional differentiability of $P_\calC$ \emph{everywhere} and in \emph{all} directions because we consider only projected trajectories that come with a priori guarantees on the existence of their time derivative.

\begin{lemma}\label{lem:proj_traj_basic}
	Consider an $\alpha$-prox-regular set $\calC \subset \bbR^n$ and an absolutely continuous map $x: [0, T] \rightarrow  \calC + \frac{1}{2\alpha} \interior \bbB$ for some $T> 0$.
	Then, $\o{x} :=  P_\calC \circ x$ is single-valued and absolutely continuous.
	Furthermore, $\d{\o{x}}(t)$ and $\d{x}(t)$ exist and satisfy $\d{\o{x}}(t) = DP_\calC(x(t); \d{x}(t))$ for almost all $[0, T]$.
\end{lemma}

\begin{proof}
	Since $\calC$ is $\alpha$-prox-regular, \cref{lem:lip_prox} guarantees that $P_\calC$ is Lipschitz in every closed neighborhood of $\calC$ that is a subset of $\calC + \frac{1}{2 \alpha} \interior \bbB$ (in particular $x([0, T])$ is compact by continuity of $x$). Since the composition of a Lipschitz map and an absolutely continuous function is absolutely continuous~\cite[Ex. 6.44]{roydenRealAnalysis1988}, it follows that $\o{x}$ is absolutely continuous and hence differentiable almost everywhere.

	Since $x$ and $\o{x}$ are differentiable everywhere except on zero measure sets $\Xi_x, \Xi_{\o{x}} \subset [0, T]$, respectively, it follows that $\dot{x}(t)$ and $\dot{\o{x}}(t)$ both exist except on the zero-measure set $\Xi_x \cup \Xi_{\o{x}}$ and $\d{\o{x}}(t) = DP_\calC(x(t); \d{x}(t))$ holds by definition of the time derivative of $\o{x}$.
\end{proof}

\begin{remark}
	The existence of $\d{\o{x}}(t)$ is in general independent of the existence of $\d{x}(t)$. On one hand, even if $\d{x}(t)$ exists, $\d{\o{x}}(t)$ might not exist because of a lack of directional differentiability. On the other hand, $\d{\o{x}}(t)$ might exist, even though $\d{x}(t)$ does not. This can occur, for instance, if $\calC = \{ 0 \}$ in which case $\o{x} \equiv 0$ is trivially differentiable everywhere.
\end{remark}

\section{Problem Formulation \& Main Results}\label{sec:main_prob}

Throughout the rest of the paper we consider the problem~\eqref{eq:basic_prob}, albeit under the following assumption:

\begin{assumption}\label{ass:basic}
	Let $\calC \subset \bbR^n$ be $\alpha$-prox-regular. Further, let $\Phi: \bbR^n \rightarrow \bbR$ be differentiable in a neighborhood of $\calC$ with compact sublevel sets $\calS_\ell := \{ x \in \calC \, | \, \Phi(x) \leq \ell \} $.
\end{assumption}
Under \cref{ass:basic}, $x^\star \in \calC$ is a \emph{critical point of~\eqref{eq:basic_prob}} (i.e., \emph{1st-order optimal}) if $\nabla \Phi(x^\star) \in - N_{x^\star} \calC$. Namely, local optimizers of~\eqref{eq:basic_prob} are critical~\cite[Th. 6.12]{rockafellarVariationalAnalysis2009}.

Instead of the dynamics~\eqref{eq:aw_sys}, we consider the inclusion
\begin{align}\label{eq:grad_aw_approx}
    \dot x \in F(x) := -  \nabla \Phi(P_\calC(x)) -  \tfrac{1}{K} \left( x - P_\calC(x) \right)
\end{align}
since $P_\calC$ is not necessarily single-valued outside $\calC + \frac{1}{2\alpha} \interior \bbB$. However, we will not concern ourselves with potential solutions outside of $\calC + \frac{1}{2\alpha} \interior \bbB$. Instead, we define the sets of admissible initial conditions (which we later show to be invariant) as
\begin{align*}
		\calC_\ell := \left\{ x \in \calC + \tfrac{1}{2\alpha} \interior \bbB \, | \, P_\calC(x) \in \calS_\ell \right\}
\end{align*}
which is the preimage of $\calS_\ell$ restricted to $\calC + \tfrac{1}{2\alpha} \interior \bbB$.

Our first main result guarantees that there always exists $K>0$ such that the projected trajectories of the anti-windup gradient flow~\eqref{eq:grad_aw_approx} converge to the critical points of~\eqref{eq:basic_prob}, although, $K$ may depend on the choice on $\ell$ and thereby on the set of initial conditions.

\begin{theorem}\label{thm:main}
    Under \cref{ass:basic} and given $\ell \in \bbR$, there exists $K^\star > 0$ such that~\eqref{eq:grad_aw_approx} admits a complete solution $x: [0, \infty) \rightarrow \calC_\ell$ for all $K \in (0, K^\star)$ and all initial conditions $x(0) \in \calC_\ell$.

    Further, for any such solution, the projected trajectory $\o{x} := P_\calC \circ x$ converges to the set of critical points of~\eqref{eq:basic_prob}.
\end{theorem}

\cref{thm:main} also applies to convex $\calC$ since convex sets are $\alpha$-prox-regular for any $\alpha > 0$. Nevertheless, we derive a stronger result that does not restrict the choice of initial condition or~$K$.

\begin{theorem}\label{thm:main2}
    If \cref{ass:basic} holds and $\calC$ is closed convex,~\eqref{eq:grad_aw_approx} admits a complete solution $x$ for all $K > 0$ and all $x(0) \in \bbR^n$.

    Further, for any such solution, the projected trajectory $\o{x} := P_\calC \circ x$ converges to the set of critical points of~\eqref{eq:basic_prob}.
\end{theorem}

If, in addition, $\Phi$ is convex, we find ourselves in the simplified setup of \cref{sec:simp_prob_form}. In this case, clearly, convergence is to the set of global optimizers of~\eqref{eq:basic_prob}.

\section{Proof of \cref{thm:main}}\label{sec:main_proof}

We apply \cref{thm:invar} by showing that $t \mapsto \Phi(P_\calC(x(t)))$ is non-increasing along any solution  of~\eqref{eq:grad_aw_approx}. Then, we prove that the limit set contains only critical points of~\eqref{eq:basic_prob}.

Throughout this section (and the next) we use the notation $\o{x} := P_\calC(x)$ for points and $\o{x} := P_\calC \circ x$ for trajectories.

Prox-regularity of $\calC$ and continuity of $\nabla \Phi$
guarantee the existence of solutions of~\eqref{eq:grad_aw_approx} in a neighborhood of~$\calC$:

\begin{lemma}\label{lem:exist}
    Under \cref{ass:basic}, there exists a solution of~\eqref{eq:grad_aw_approx} for every initial condition $x(0) \in \calC + \frac{1}{2\alpha} \interior \bbB$. More precisely, there exists a differentiable function $x: [0, T] \rightarrow \calC + \frac{1}{2\alpha} \interior \bbB$ for some $T > 0$ that satisfies for all $t \in [0, T]$
    \begin{align*}
        \dot x(t) = -  \nabla \Phi(P_\calC(x(t))) + \tfrac{1}{K} \left( x(t) - P_\calC(x(t)) \right) \, .
    \end{align*}
\end{lemma}

\begin{proof}
    From \cref{lem:lip_prox} it follows that $P_\calC$ is single-valued and continuous for all $x \in \calC + \frac{1}{2\alpha} \interior \bbB$. Further, since $\nabla \Phi$ is continuous, $F$ is continuous. Hence, standard results for continuous ODEs guarantee the existence of a local solution for every initial condition on the open set $\calC + \frac{1}{2\alpha} \interior \bbB$.
\end{proof}

\subsection{Convergence to Invariant Set}

To show that $\Phi$ is non-increasing along projected trajectories of~\eqref{eq:grad_aw_approx} we use the lemmas in \cref{sec:dir_deriv}. Further, to apply \cref{thm:invar} we need to show that (unprojected) trajectories of~\eqref{eq:grad_aw_approx} are complete and bounded, which is possible, in general, only for small enough  $K$ (unless $\calC$ is convex).

\begin{lemma}\label{lem:nonincreas}
	Let \cref{ass:basic} hold. Given a solution $x : [0, T] \rightarrow \calC + \frac{1}{2\alpha} \interior \bbB$ of~\eqref{eq:grad_aw_approx} for some $T> 0$, the map $t \mapsto \Phi(\o{x}(t))$ is non-increasing for all $t \in [0, T]$.
\end{lemma}

\begin{proof}
	\cref{lem:proj_traj_mono,lem:orth_prop,lem:proj_traj_basic} yield, for almost all $t \in [0, T]$,
	\begin{align*}
		\tfrac{d}{dt} \Phi(\o{x}(t))
		 & = \left\langle \nabla \Phi(\o{x}(t)), \d{\o{x}}(t) \right\rangle                           \\
		 & = \left\langle \nabla \Phi(\o{x}(t)) + \tfrac{1}{K}(x(t) - \o{x}(t)), \d{\o{x}}(t) \right\rangle \\
		 & = \left\langle - \d{x}(t), \d{\o{x}}(t) \right\rangle \leq 0 \, 
	\end{align*}
	and we conclude that $\Phi \circ \o{x}$ is non-increasing.
\end{proof}

\begin{lemma}\label{lem:cl_bounded}
	Under \cref{ass:basic}, $\calC_\ell$ is bounded $\forall\ell\in\bbR$.
\end{lemma}

\begin{proof}
	The set $\calC_\ell$ is as the preimage of $\calS_\ell$ under $P_\calC$ restricted to $\calC + \frac{1}{2\alpha} \interior \bbB$.
	From \cref{lem:proj_preimag} it follows that for any $x \in \calC$ we have $P_\calC^{-1}(x) = x + N_x \calC \cap \frac{1}{2\alpha} \interior \bbB \subset x + \frac{1}{2\alpha} \interior \bbB$.  Since $\calS_\ell$ is compact, $\calC_\ell \subset \calS_\ell + \frac{1}{2\alpha} \interior \bbB$ is bounded.
\end{proof}
\begin{proposition}\label{prop:comp_bound}
	Let \cref{ass:basic} hold. Given $\ell \in \bbR$, there exists $K^\star > 0$ such that~\eqref{eq:grad_aw_approx} admits a complete solution $x : [0, \infty) \rightarrow \calC_\ell$ for every $x(0) \in \calC_\ell$ and for all $K \in (0, K^\star)$.
\end{proposition}

\begin{proof}
	First, note that \cref{lem:exist} guarantees the existence of a (local) solution $x :[0, T] \rightarrow \calC + \frac{1}{2\alpha} \interior \bbB$ for any initial condition $x(0) \in \calC_\ell \subset \calC + \frac{1}{2\alpha} \interior \bbB$ and for some $T> 0$.

	Since \cref{lem:nonincreas} guarantees that $\Phi(\o{x}(t)) \leq \Phi(\o{x}(0)) \leq \ell$ for all $t \in [0, T]$, it follows that $\o{x}(t) \in \calS_\ell$ for all $[0, T]$.

	By compactness of $\calS_\ell$, there exists $M > 0$, such that $\| \nabla \Phi(y)\| \leq M$ for all $y \in \calS_\ell$. Now, consider the Lie derivative of $d_\calC^2$ along~\eqref{eq:grad_aw_approx}. For $x \in \calC_\ell$ we have
	\begin{equation*}
		\begin{split}
			\calL_{F} d^2_\calC(x) & = \left\langle x - \o{x}, - \nabla \Phi(\o{x}) - \tfrac{1}{K}(x - \o{x}) \right\rangle \\
		                         & \leq d_\calC(x) \| \nabla \Phi(\o{x})\|  - \tfrac{1}{K} d^2_\calC(x)                                 \\
		                         & \leq d_\calC(x) \left( M - \tfrac{1}{K} d_\calC(x) \right) \, .
		\end{split}
	\end{equation*}
	It follows that $\calL_{F} d^2_\calC(x) < 0$ for all $x \in \calC_\ell$ for which $d_\calC(x) > KM$. In particular, if $K < K^\star  := \frac{1}{2\alpha M}$, any solution $x$ of~\eqref{eq:grad_aw_approx} starting in $\calC_\ell$ cannot leave the neighborhood $\calC + \frac{1}{2\alpha} \interior \bbB$ on which $P_\calC$ is single-valued. In addition, $\o{x} = P_\calC(x)$ remains in $\calS_\ell$. Hence $\calC_\ell$ is invariant. Together with the boundedness of $\calC_\ell$, finite-time escape is precluded and thus guaranteeing the existence of a complete solution.
\end{proof}

\begin{proposition}\label{prop:invar}
    Under \cref{ass:basic} any complete solution $x: [0, \infty) \rightarrow \calC_\ell$ of~\eqref{eq:grad_aw_approx} converges to the largest weakly invariant subset $\calS$ of $\calM := \cl \{ x  \in \calC_\ell \, | \, DP_\calC(x; F(x)) = 0 \}.$
\end{proposition}

\begin{proof}
	Note that $\Phi \circ P_\calC$ is continuous on $\calC + \frac{1}{2\alpha} \interior \bbB$ by continuity of $\Phi$ and \cref{lem:lip_prox}. Hence, to apply \cref{thm:invar}, let $V: \bbR^n \rightarrow \bbR$ be any continuous function such that $V(y) = \Phi(P_\calC(y))$ for all $y \in \calC + \frac{1}{2\alpha} \interior \bbB$. Further, let $\calU := \calC_\ell$. The trajectory $x$ is complete by assumption and bounded by Lemma~\ref{lem:cl_bounded}. Hence, according to \cref{thm:invar}, $x$ converges to the largest weakly invariant subset of $V^{-1}(r) \cap \calU \cap \cl u^{-1}(0) $ for some $r$ and where we have
	\begin{align*}
		u^{-1}(0)
		 & = \{ x   \in \calU \, | \, \left\langle F(x), DP_\calC(x; F(x)) \right\rangle  = 0 \} \\
		 & = \{ x   \in \calU \, | \, DP_\calC(x; F(x)) = 0 \} \, ,
	\end{align*}
	where the second equality follows from \cref{lem:proj_equi}.
\end{proof}

It is important to note that $\cl \{ x   \in \calC_\ell \, | \, DP_\calC(x; F(x)) = 0 \}$ is not, in general, invariant itself. There can exist compact intervals $[t_1, t_2]$ on which $\o{x}$ is constant (and hence $\d{\o{x}}(t) = 0$ for all $t \in [t_1, t_2]$), but on which $x$ is not stationary. For example, in Fig.~\ref{fig:proj_traj}, this is the case when $\o{x}(t) = P_\calC(x(t))$ is stuck in one of the vertices of the feasible polyhedron $\calC$, while $x$ is evolving outside of $\calC$, moving ``around the corner''.

\subsection{Characterization of Invariant Limit Set}

Next, we show that the largest weakly invariant subset $\calS$ in \cref{prop:invar} is equivalent to the critical points of~\eqref{eq:basic_prob}. 

\begin{lemma}\label{lem:invar_proj}
	Consider the setup of \cref{prop:invar} and let  $x : [0, \infty) \rightarrow \calS$ be a complete solution of~\eqref{eq:grad_aw_approx} evolving on the weakly invariant set $\calS$. Then, $\o{x}(t) = \o{x}(0)$ holds.
\end{lemma}

\begin{proof}
	Since $\o{x}$ is absolutely continuous, it follows that $\o{x}(\tau) - \o{x}(0) = \int_0^\tau \d{\o{x}}(t) dt$. However, $\d{\o{x}}(t) = 0$ holds for almost all $t \geq 0$ since, by invariance, $\o{x}(t) \in \calS \subset \calM$ and therefore $\o{x}(\tau) = \o{x}(0)$.
\end{proof}

\begin{proposition}\label{prop:opt_cond}
	Consider the setup of \cref{prop:invar}. Then, every $x^\star \in P_\calC(\calS)$ is a critical point of~\eqref{eq:basic_prob}.
\end{proposition}

\begin{proof}
	Consider a trajectory $x: [0, \infty) \rightarrow \calS$ evolving on the weakly invariant set $\calS$. By \cref{lem:invar_proj}, we have that $\o{x}(t) = \o{x}(0) =: y$ for all $t \geq 0$. Therefore, $x$ evolves on the preimage $P_\calC^{-1}(y)$ which, using \cref{lem:proj_preimag}, is given by $y + N_y \calC$. In other words, $x(t) \in y + N_y \calC$ for all $t \geq 0$. In particular, $x$ satisfies $\dot x(t) = - \nabla \Phi(y) - \tfrac{1}{K}(x(t) - y)$
	for all $t \geq 0$. Thus, $x$ is also the solution of an asymptotically stable linear system and converges to a point $\hat{x}$ such that $- \nabla \Phi(y) = \frac{1}{K}(\hat{x} - y) \in N_y \calC$ and $P_\calC (\hat{x}) = y$ hold. In other words, $y$ is a critical point.
\end{proof}

\cref{thm:main} now follows directly since \cref{prop:comp_bound} yields the existence of a complete solution and \cref{prop:invar,prop:opt_cond} guarantee the convergence of $\o{x}$ to the set of critical points.

\section{Proof Sketch for \cref{thm:main2}}

\cref{thm:main2} does not directly derive from \cref{thm:main} by letting $\alpha \rightarrow 0^+$, because $\lim_{\alpha \rightarrow 0^+} \calC_\ell$ is not bounded. Instead, we need to adapt \cref{prop:comp_bound} as follows:
\begin{proposition}\label{prop:comp_bound_cvx}
	Let \cref{ass:basic} hold and let $\calC$ be convex. Then~\eqref{eq:grad_aw_approx} admits a complete and bounded solution for every initial condition $x(0) \in \bbR^n$ and all $K > 0$.
\end{proposition}

\begin{proof}
    The proof is analogous to the proof of \cref{prop:comp_bound}. In particular, we have $\calL_F d_\calC^2 (x) < 0$ for all $x \in \calC_\ell$ for which $d_\calC(x) > KM$. However, since $P_\calC$ is globally single-valued, $K$ does not need to be chosen small enough to guarantee the invariance of a neighborhood $\calC + \frac{1}{2\alpha}\interior \bbB$. Instead, we have
	\begin{align*}
		x(t) \in \calC + \gamma \bbB \qquad \forall t \in [0, T]
	\end{align*}
	for all $t \geq 0$ with $\gamma > \max \{ KM, d_\calC(x(0)) \}$. More precisely,
	\begin{align*}
	    x(t) \in \calC^\gamma_\ell := \{ y \in  \calC + \gamma \bbB \, | \, P_\calC(y) \in \calS_\ell \} \, ,
	\end{align*}
	for all $t \geq 0$ and where $\ell := \Phi(P_\calC(x(0)))$. This follows from Lemma~\ref{lem:nonincreas} since $t \mapsto \Phi(P_\calC(x(t)))$ is non-increasing. Using the same argument as for Lemma~\ref{lem:cl_bounded}, we can show that $\calC^\gamma_\ell$ is bounded.
\end{proof}

Finally, \cref{prop:invar,prop:opt_cond} can be adapted using $ \calC^\gamma_\ell$ instead of $\calC_\ell$ and \cref{thm:main2} follows similarly to \cref{thm:main}.

\section{Conclusions}\label{sec:concl}

We have studied the convergence properties of anti-windup gradient flows and established semi-global convergence of projected trajectories for prox-regular domains. For convex domains convergence is global for any anti-windup gain. Using properties of projected trajectories we have hence been able generalize~\cite[Th. 6.4]{hauswirthAntiWindupApproximationsOblique2020} for gradient flows. However, it remains open whether the same analysis can also yield stronger convergence results for anti-windup approximations of other optimization dynamics such as variable-metric gradient using oblique projections \cite{hauswirthProjectedDynamicalSystems2018} or saddle-point flows. For preliminary results in these directions, as well as simulation results, the reader is referred to \cite{hauswirthAntiWindupApproximationsOblique2020}.

\bibliography{bibliography}
\bibliographystyle{IEEEtran}

\end{document}